\documentclass{amsart}
\usepackage{amsfonts,amssymb,amscd,amsmath,enumerate,verbatim,calc}
\usepackage[all]{xy}

\newcommand{\CM}{Cohen-Macaulay}

\newcommand{\wrt}{with respect to}

\newcommand{\m}{\mathfrak{m} }

\newcommand{\q}{\mathfrak{q} }

\newcommand{\FF}{\mathcal{T}}
\newcommand{\GG}{\mathcal{G}}

\newcommand{\C}{\mathbf{C} }

\newcommand{\rt}{\rightarrow}

\newcommand{\ov}{\overline}

\newcommand{\Ass}{\operatorname{Ass}}

\newcommand{\height}{\operatorname{height}}

\newcommand{\Supp}{\operatorname{Supp}}

\newcommand{\injdim}{\operatorname{injdim}}

\newcommand{\Hom}{\operatorname{Hom}}
\newcommand{\Ext}{\operatorname{Ext}}

\newcommand{\Spec}{\operatorname{Spec}}

\theoremstyle{plain}

\newtheorem{theorem}{Theorem}[section]

\newtheorem{lemma}[theorem]{Lemma}
\newtheorem{proposition}[theorem]{Proposition}

\theoremstyle{definition}

\newtheorem{s}[theorem]{}
\newtheorem{remark}[theorem]{Remark}

\theoremstyle{remark}

\begin{document}

\title{On injective resolutions of local cohomology modules}
 \author{Tony J. Puthenpurakal}
\date{\today}
\address{Department of Mathematics, Indian Institute of Technology Bombay, Powai, Mumbai 400 076, India}
\email{tputhen@math.iitb.ac.in}

\subjclass{Primary 13D45; Secondary 13D02, 13H10 }
\keywords{local cohomology, injective resolutions}

\begin{abstract}
Let $K$ be a field of characteristic zero   and let $R = K[X_1,\ldots,X_n]$. Let $I$ be an ideal in $R$ and let   $M = H^i_I(R)$ be the $i^{th}$-local
cohomology module of $R$ with respect to $I$.
 Let $c = \injdim M$. We prove that if $P$
is a prime ideal in $R$ with Bass number $\mu_c(P,M) > 0$ then $P$ is a maximal ideal in $R$.
\end{abstract}

\maketitle
\section{Introduction}
Throughout this paper $R$ is a commutative Noetherian ring. If $M$ is an $R$-module and  $Y$ be a locally closed subscheme of $\Spec(R)$, we denote by $H^i_Y(M)$ the $i^{th}$-local cohomology module of $M$ with support in $Y$. If $Y$ is closed in $\Spec(R)$ with defining ideal $I$ then $H^i_Y(M)$ is denoted by  $H^i_I(M)$. 
 
In a remarkable paper, \cite{HuSh}, Huneke and Sharp proved that if $R$ is a regular ring containing a field of characteristic $p > 0$, and $I$ is an ideal in $R$ then the local cohomology modules of $R$ \wrt \ $I$ have the following properties:
\begin{enumerate}[\rm(i)]
\item
$H^j_{\m}(H^i_I(R))$ is injective, where $\m$ is any maximal ideal of $R$.
\item
$\injdim_R H^i_I(R) \leq \dim \Supp H^i_I(R)$.
\item
The set of associated primes of $H^i_I(R)$ is finite.
\item
All the Bass numbers of $H^i_I(R)$ are finite.
\end{enumerate}
Here $\injdim_R H^i_I(R)$ denotes the injective dimension of $H^i_I(R)$. Also $\Supp M = \{ P \mid  M_P \neq 0 \ \text{and $P$ is a prime in $R$}\}$ is the support of an $R$-module $M$. The $j^{th}$ Bass number of an $R$-module $M$ with respect to a prime ideal $P$ is defined as $\mu_j(P,M) = \dim_{k(P)} \Ext^j_{R_P}(k(P), M_P)$ where $k(P)$ is the residue field of $R_P$.

 In another remarkable paper, for regular rings in characteristic zero, Lyubeznik was able to establish the above properties  
for a considerably larger class of functors   than just the local cohomology modules, see  \cite{Lyu-1}. We call such functors as \textit{Lyubeznik functors}, see section two  for details. 
If $\FF$ is a Lyubeznik functor on $Mod(R)$ then $\FF(R)$ satisfies the following properties:
  \begin{enumerate}[\rm(i)]
\item
$H^j_{\m}(\FF(R))$ is injective, where $\m$ is any maximal ideal of $R$.
\item
$\injdim_R \FF(R) \leq \dim \Supp \FF(R)$.
\item
For every maximal ideal $\m$, the number of  associated primes of $\FF(R)$ contained in $\m$
 is finite.
\item
All the Bass numbers of $\FF(R)$ are finite.
\end{enumerate}
We should note that if $R = K[X_1,\ldots,X_n]$ then the number of associate primes of $\FF(R)$ is finite.

 The results of Lyubeznik for characteristic zero raised the question of whether the results (i)–-(iv) of Huneke and Sharp (in characteristic $p>0$) could be extended to this larger class of functors. In \cite{Lyu-2}, Lyubeznik proves it.
 
 If $M$ is a finitely generated module over a \CM \ ring $R$ and say $M$ has finite injective dimension $d =  \dim R$, then it is elementary to prove that if $\mu_d(P,M) > 0$ then $P$ is a maximal ideal in $R$, use \cite[3.1.13]{BH}. This fails for modules which are not finitely generated, for instance consider the injective hull $E(R/P)$ of  $R/P$  where $P$ is a  prime ideal which is not maximal. 
 
 Our main result is 
\begin{theorem}\label{main}
Let $K$ be a field of characteristic zero and let $R = K[X_1,\ldots,X_n]$. Let $\FF$ be a Lyubeznik functor on $Mod(R)$. Suppose $ \injdim \FF(R) = c$. If $P$ is a prime ideal in $R$ with Bass number $\mu_c(P, \FF(R)) > 0 $ then $P$ is a maximal ideal of $R$.
\end{theorem}
As an aside we note that to best of our knowledge this is the first result whose proof  uses the fact that $\Ass \FF(R)$ is finite for any Lyubeznik functor $\FF$.

A natural question is what can we say about $\mu_c(\m,\FF(R))$ as $\m$ varies over maximal ideals in $R$.
Our next result is essentially only an observation.
\begin{proposition}\label{main-2}
Let $K$ be an algebraically closed field of characteristic zero and let $R = K[X_1,\ldots,X_n]$. Let $\FF$ be a Lyubeznik functor on $Mod(R)$. Suppose $ \injdim \FF(R) = c$. Then for all $i = 0,\ldots, c$; the set 
$$\{ \mu_i(\m, \FF(R)) \mid \m  \ \text{a maximal ideal of }\ R \}$$
 is bounded.
\end{proposition}
The surprising thing about Proposition \ref{main-2} is that I do not know whether such a result holds for finitely generated modules over $R$.
 
A natural question is whether the results \ref{main} and \ref{main-2} hold in characteristic $p > 0$.  Although we expect this to be true; our techniques do not work in positive characteristic.  We are only able to extend Propostion \ref{main-2} to a subclass of Lyubeznik functors, see \ref{char-p}.

We now describe in brief the contents of this paper. In section two we define Lyubeznik functors and also a few preliminary results on holonomic modules which we need. In section three we discuss two Lemmas which will help in proving Theorem \ref{main}. In section four we prove Theorem \ref{main}. Finally in section five we prove Proposition \ref{main-2}.
\section{Preliminaries}
In this section we define Lyubeznik functors. We also prove a result on holonomic modules which we need. 

\s \textit{Lyubeznik functors:}\\
Let $R$ be a commutative Noetherian ring and let $X = \Spec(R)$. Let $Y$ be a locally closed subset of $X$. If $M$ is an $R$-module and  $Y$ be a locally closed subscheme of $\Spec(R)$, we denote by $H^i_Y(M)$ the $i^{th}$-local cohomology module of $M$ with support in $Y$.  Suppose 
 $Y = Y_1 \setminus Y_2$ where $Y_2 \subseteq Y_1$ are two closed subsets of $X$ then we have an exact sequence of functors
 \[
 \cdots \rt H^i_{Y_1}(-) \rt H^i_{Y_2}(-) \rt H^i_Y(-) \rt H^{i+1}_{Y_1}(-) \rt .
 \]
 A Lyubeznik functor $\FF$ is any functor of the form $\FF = \FF_1\circ \FF_2 \circ \cdots \circ \FF_m$ where every functor $\FF_j$  is either $H^i_Y(-)$ for some locally closed subset of $X$ or the kernel, image or
cokernel of some arrow in the previous long exact sequence for closed
subsets $Y_1,Y_2$ of $X$  such that $Y_2 \subseteq Y_1$.

We need the following result from \cite[3.1]{Lyu-1}.
\begin{proposition}\label{flat-L}
 Let $\phi \colon R \rt S$ be a flat homomorphism of Noetherian rings. Let $\FF$ be a 
 Lyubeznik functor on $Mod(R)$. Then there exists a Lyubeznik functor $\widehat{\FF}$ on $Mod(S)$ and isomorphisms $\widehat{\FF}(M\otimes_R S) \cong \FF(M)\otimes_R S$ which is functorial in $M$.
\end{proposition}
 
\s  \textit{Lyubeznik functors and holonomicity:}\\
 Let $K$ be a field of characteristic zero. Let $S = K[[X_1,\ldots,X_n]]$. Let $D$ be the ring of $K$-linear differential operators on $S$. Let $\FF $ be a Lyubeznik functor on $Mod(S)$. If $M$ is any holonomic $D$-module then $\FF(M)$ is a holonomic $D$-module; see \cite[2.2d]{Lyu-1}. In particular $\FF(S)$ is a holonomic $D$-module. 
 
 Let $R = K[X_1,\ldots,X_n]$ and let $A_n(K)$ be the $n^{th}$-Weyl algebra over $K$. 
 Let $\FF$ be a Lyubeznik functor on $Mod(R)$. If $M$ is any holonomic $A_n(K)$-module then $\FF(M)$ is a holonomic $A_n(K)$-module; (the proof in \cite[2.2d]{Lyu-1} can be modified to prove this result). In particular $\FF(R)$ is a holonomic $A_n(K)$-module.
 
 \begin{remark}
 In \cite{B} holonomic $A_n(K)$-modules are called modules belonging to the Bernstein class.
 \end{remark}
\s\label{d-mod} Let $k$ be a field of characteristic zero and let $S = k[[Y_1,\ldots,Y_n]]$. Let $D$ be the ring of $k$-linear differential operators on $S$. Let $C$ be a simple holonomic $D$-module. Notice $\Ass_S C = \{ P \}$ for some prime $P$ in $S$. Also $C$ is $P$-torsion; see \cite[3.3.16-17]{B}. It follows from \cite[p.\ 109, lines 3-6]{B} that  there exists $h \in (S/P)$ non-zero such that $\Hom_S(S/P, C)_h$ is a finitely generated $(S/P)_h$ module. Let $g$ be a pre-image of $h$ in $S$. Then clearly 
$\Hom_S(S/P,C)_g$ is a finitely generated $S_g$-module. We  now generalize this result.
\begin{proposition}\label{fin-gen}
(with hypotheses as in \ref{d-mod}) Let $M$ be a holonomic $D$-module. Assume $\Ass_S M = \{ P \}$ and $M$ is $P$-torsion. Then there exists $h \in S\setminus P$ such that $\Hom_S(S/P, M)_h$ is finitely generated as a $S_h$-module.
\end{proposition}
\begin{proof}
Let $ 0 = M_0 \varsubsetneq M_1\varsubsetneq M_2\varsubsetneq \cdots \varsubsetneq M_{n-1} \varsubsetneq M_n = M$ be a filtration of $M$ with $M_i/M_{i-1}$ simple $D$-module for $i = 1,\ldots,n$. By induction on $i$ we prove that there exists $h_i \in S\setminus P$ such that $\Hom_S(S/P, M_i)_{h_i}$ is finitely generated as a $S_{h_i}$-module.

For $i = 1$ note that $M_1$ is a simple holonomic $D$-module. Also $\Ass_S M_1 \subseteq \Ass_S M = \{ P \}$. Then by \ref{d-mod} we get the required assertion.
We assume the result for $i = r$ and prove it for $i = r + 1$.
Say $\Hom_S(S/P, M_r)_{h_r}$ is a finitely generated $S_{h_r}$-module. We consider the following two cases.

Case 1 : $\Ass_S M_{r+1}/M_r = \{ P \}$. \\  By \ref{d-mod} there exists $g_r \in S \setminus P$ such that $\Hom_S(S/P, M_{r+1}/M_r)_{g_r}$ is finitely generated $S_{g_r}$-module. 
Consider the exact sequence
\[
0 \rt \Hom_S(S/P, M_r) \rt \Hom_S(S/P, M_{r+1}) \rt \Hom_S(S/P, M_{r+1}/M_r).
\]
Localize at $h_{r+1} = h_rg_r \in S \setminus P$. Notice
\begin{enumerate}
\item
$\Hom_S(S/P, M_r)_{h_{r+1}} = \left(\Hom_S(S/P,M_r)_{h_r}\right)_{g_r}$ is finitely generated as a $S_{h_{r+1}}$-module.
\item
$\Hom_S(S/P, M_{r+1}/M_r)_{h_{r+1}} = \left( \Hom_S(S/P, M_{r+1}/M_r)_{g_r} \right)_{h_r}$ is  finitely generated as a $S_{h_{r+1}}$-module.
\end{enumerate}
It follows that $\Hom_S(S/P, M_{r+1})_{h_{r+1}}$ is finitely generated as a $S_{h_{r+1}}$-module.

Case 2: $\Ass_S M_{r+1}/M_r = \{ Q \}$ with $Q \neq P$. \\ As $M$ is $P$-torsion we have that $Q\supsetneq P$. Take $g \in Q \setminus P$. Then $(M_{r+1}/M_r)_g = 0$. So $\Hom_S(S/P, M_{r+1}/M_r)_g = 0$. Put $h_{r+1} = h_rg \in S \setminus P$. Then note that 
\[
\Hom_S(S/P, M_{r+1})_{h_{r+1}} \cong \Hom_S(S/P, M_r)_{h_{r+1}} = \left(\Hom_S(S/P,M_r)_{h_r}\right)_{g},
\]
 is finitely generated as a $S_{h_{r+1}}$-module.
Thus by induction we get that there exists $h \in S \setminus P$ such that $\Hom_S(S/P, M)_h$ is finitely generated as a $S_h$-module.
\end{proof}

\s Finally we need the following well-known result regarding non-singular locus of affine domains. 
\begin{theorem}\label{locus}
Let $A$ be an affine domain, finitely generated over a  perfect field $k$. Then
\begin{enumerate}[\rm (1)]
\item
The non-singular locus  of $A$ is non-empty and an open subset of $\Spec(A)$.
\item
There exists a maximal ideal $\m$ of $A$ with $A_\m$ regular local.
\item
If $\dim A \geq 1$ then there exists infinitely many maximal ideals of $A$ with $A_\m$ regular local.
\item
Suppose  $\dim A \geq 2$ and let $f \in A$. Then there exists a maximal ideal $\m$ of $A$ with $f \notin \m $ and $A_\m$-regular local.
\end{enumerate}
\end{theorem}

\section{Two Lemma's}
In this section we establish two lemma's which will enable us to prove our main result. Let $K$ be a field of characteristic zero and let $P$ be a prime ideal of height in $R = K[X_1,\ldots,X_n]$. Let $E(R/P)$ denote the injective hull of $R/P$.
Recall that $E(R/P) = H^g_P(R)_P$. It follows that $E(R/P)$ is a $A_n(K)$-module and the natural inclusion $H^g_P(R) \rt E(R/P)$ is $A_n(K)$-linear.

\begin{lemma}\label{lemma1} Let $K$ be a field of characteristic zero and let $P$ be a prime ideal of height $n-1$ in $R = K[X_1,\ldots,X_n]$. Then $E(R/P)$ is not a holonomic $A_n(K)$-module. 
\end{lemma}
\begin{proof}
Suppose if possible $E(R/P)$ is a holonomic $A_n(K)$-module.  We have an exact sequence  of $A_n(K)$-modules
\[
0 \rt H^{n-1}_P(R) \rt E(R/P) \rt C \rt 0.
\]
As $E(R/P)$ is holonomic we have that $C$ is also a holonomic $A_n(K)$-module. Notice $C_P = 0$. It follows that $C$ is supported at only finitely many maximal ideals of $R$, say $\m_1,\ldots,\m_r$. By Theorem \ref{locus}(3) there exists a maximal ideal $\m$ of $R$ such that $\m \neq \m_i$ for all $i$ and $(R/P)_\m$ is regular local.
Note $H^{n-1}_P(R)_\m = E(R/P)_\m$ as $C_\m = 0$. If $\m R_\m = (z_1,\ldots,z_n)$ then as $R_\m/P R_\m$ is regular we may assume that $P R_\m = (z_1,\ldots,z_{n-1})$. In particular $H^n_{PR_\m}(R_\m) = 0$.

Let $f \in \m R_\m \setminus P R_\m$. Note that we have an exact sequence
\[
0 \rt H^{n-1}_{PR_\m}(R_\m) \rt H^{n-1}_{PR_\m}(R_\m)_f \rt H^n_{(PR_\m ,f)}(R_\m) \rt H^n_{PR_\m}(R_\m) = 0
\] 
As $H^{n-1}_{PR_\m}(R_\m) = H^{n-1}_P(R)_\m = E(R/P)_\m$ it follows that the first map in the above exact sequence 
is an isomorphism. It follows that $H^n_{(PR_\m ,f)}(R_\m) = 0$. This contradicts Grothendieck's non-vanishing theorem as $\sqrt{(PR_\m, f)} = \m R_\m$.
\end{proof}

Our next  result is

\begin{lemma}\label{lemma2}
Let $K$ be a field of characteristic zero and let $P$ be a height $g$ prime in $R = K[X_1,\ldots,X_n]$ with $g \leq n-2$. Suppose $\m$ is a maximal ideal in $R$ with $(R/P)_\m$ a regular local ring. Let $\FF$ be a Lyubeznik functor on $Mod(R_\m)$. Then $\FF(R_\m) \neq E(R/P)_\m^c$ for any $c > 0$.
\end{lemma}
\begin{proof}
Suppose if possible $\FF(R_\m) = E(R/P)_\m^c$ for some $c > 0$. Let $\widehat{R_\m}$ be the completion of $R_\m$ at $\m R_\m$. Note $\widehat{R_\m} = K^\prime[[Z_1,\ldots,Z_n]]$ where $K^\prime \cong R/\m$. Let $D$ be the ring of $K^\prime$-linear differential operators on $\widehat{R_\m}$. Note by \ref{flat-L} there exists a Lyubeznik functor $\widehat{\FF}$ on $Mod(\widehat{R_\m})$  such that $ \widehat{\FF}( \widehat{R_\m}) = \FF(R_\m)\otimes \widehat{R_\m}$.  In particular $E(R/P)_\m^c \otimes \widehat{R_\m}$ is a holonomic $D$-module. So $V = E(R/P)_\m\otimes \widehat{R_\m}$ is a holonomic $D$-module. 

As $(R/P)_\m$ is regular local we may assume that $PR_\m = (Z_1,\ldots,Z_g)$. Note $n \geq g + 2$. In particular we have that $P\widehat{R_\m}$ is a prime ideal in $\widehat{R_\m}$. Notice $V$ is $P\widehat{R_\m}$-torsion. Furthermore $\Ass V =  \{ P\widehat{R_\m} \}$.  Using Proposition \ref{fin-gen} we get that  there exists $h \in \widehat{R_\m} \setminus P\widehat{R_\m}$
such that $\Hom(\widehat{R_\m}/P\widehat{R_\m}, V)_h$ is a finitely generated $(\widehat{R_\m})_h$-module.
Notice $\Hom_{R_\m}(R_\m/PR_\m, E(R/P)_\m) = k(P)$ where $k(P)$ is the quotient field of $R_\m/P R_\m$. It follows that  
\[
\Hom(\widehat{R_\m}/P\widehat{R_\m}, V) = \Hom_{R_\m}(R_\m/PR_\m, E(R/P)_\m) \otimes \widehat{R_\m}
= k(P)\otimes \widehat{R_\m}.
\]

For $\lambda \in K$ let $\q_\lambda = (Z_1,\ldots,Z_g, Z_{g+1} + \lambda Z_{g+2})$. Clearly $\q_\lambda $ is a prime ideal of height $g+1$ in $R_\m$ containing $PR_\m$. Furthermore  we have that $\q_\lambda \widehat{R_\m}$ is a prime ideal in $\widehat{R_\m}$. If $\lambda_1 \neq \lambda_2$ then it is easy to show that $\q_{\lambda_1} \neq \q_{\lambda_2}$.
Now consider $\ov{h}$, the image of $h$ in $\widehat{R_\m}/P\widehat{R_\m}$. By considering a primary decomposition of $(\overline{h})$ it follows that infinitely many $\q_\lambda \widehat{R_\m}$ do not contain $h$. Choose one such $\lambda$. Thus we have that 
$\Hom(\widehat{R_\m}/P\widehat{R_\m}, V)_{\q_\lambda \widehat{R_\m}} $ is a finitely generated 
$(\widehat{R_\m})_{\q_\lambda \widehat{R_\m}}$-module. 
Notice we have a flat local map
$(R_\m)_{\q_\lambda} \rt (\widehat{R_\m})_{\q_\lambda \widehat{R_\m}}$.
Furthermore note that
\begin{align*}
\Hom(\widehat{R_\m}/P\widehat{R_\m}, V)_{\q_\lambda \widehat{R_\m}} &= k(P)\otimes_{R_\m}\widehat{R_\m} \otimes_{\widehat{R_\m}} (\widehat{R_\m})_{\q_\lambda \widehat{R_\m}}, \\
&= k(P)\otimes_{R_\m} (\widehat{R_\m})_{\q_\lambda \widehat{R_\m}}, \\
&= k(P)\otimes_{R_\m} (R_\m)_{\q_\lambda} \otimes_{(R_\m)_{\q_\lambda}}(\widehat{R_\m})_{\q_\lambda \widehat{R_\m}}, \\
&=  k(P) \otimes_{(R_\m)_{\q_\lambda}}(\widehat{R_\m})_{\q_\lambda \widehat{R_\m}}.
\end{align*}
In the last equation we have used that $k(P)_{\q_\lambda} = k(P)$. By Proposition \ref{flat-inf-gen} we get that $k(P)$ is a finitely generated $(R_\m)_{\q_\lambda}$-module. This is a contradiction as $P(R_\m)_{\q_\lambda}$ is a non-maximal prime ideal in $(R_\m)_{\q_\lambda}$.
\end{proof}

We need the following result in the proof of Lemma \ref{lemma2}.
\begin{proposition}\label{flat-inf-gen}
Let $\phi \colon A \rt B$ be a flat local map of Noetherian local rings. Let $L$ be an $A$-module.
Then $L$ is finitely generated as a $A$-module if and only if $L \otimes_A B$ is finitely generated as a $B$-module.
\end{proposition}
\begin{proof}
If $L$ is finitely generated as a $A$-module then clearly $L\otimes_A B$ is finitely $B$-module.
Suppose now that $L$ is not a finitely generated $A$-module. Let
$$ L_1 \varsubsetneq L_2 \varsubsetneq \cdots \varsubsetneq L_n \varsubsetneq L_{n+1} \varsubsetneq \cdots $$
be a strictly ascending chain of submodules in $L$.
By faithful flatness we have that
$$ L_1\otimes B \varsubsetneq L_2\otimes B\varsubsetneq \cdots \varsubsetneq L_n\otimes B \varsubsetneq L_{n+1} \otimes B \varsubsetneq \cdots $$
is a strictly ascending chain of submodules of $L\otimes B$. It follows that $L\otimes B$ is not finitely generated.
\end{proof}

\section{Proof of the  Theorem \ref{main}}
In this section we prove our main result. We need the following easily proved fact.
\begin{proposition}\label{local}
Let $A$ be a Noetherian ring
 and let $T$ be an $A$-module. Let $f \in A$.
 Then the natural map 
 $$\eta \colon T \rt T_f  \ \text{ is injective if and only if } \  f \notin \bigcup_{P \in \Ass T}P. $$
\end{proposition}

We now give
\begin{proof}[Proof of Theorem \ref{main}]
Set $M = \FF(R)$. We prove that if $P$ is a prime ideal in $R$ and not maximal then $\mu_c(P, M) = 0$. Notice $\mu_c(P, M) = \mu_0(P, H^c_P(M))$, see \cite[1.4, 3.4(b)]{Lyu-1}. We consider two cases.

\textit{Case 1:} $\height P = n-1$.\\ Suppose if possible $\mu_0(P, H^c_P(M)) \neq 0$. Notice then $P$ is a minimal prime of $H^c_P(M)$. So if $\q \in \Ass_R H^c_P(M)$ and $\q \neq P$ then $\q$ is a maximal ideal of $R$.  In this case $\Gamma_\q(H^c_P(M)) = E(R/\q)^r$ for some $r > 0$. 
Since $\Ass H^c_P(M)$ is a finite set we can write $H^c_P(M) = L \oplus I$ \textit{as $R$-modules }where $\Ass_R L = \{ P \}$ and $I = E(R/\m_1)^{r_1}\oplus E(R/\m_2)^{r_2}\oplus \cdots \oplus E(R/\m_s)^{r_s}$ for some maximal ideals $\m_1,\ldots,\m_s$ and finite numbers $r_1,\ldots,r_s$. 
Thus $I$ is an injective $R$-module. Also note that both $L$ and $I$ are $P$-torsion. Further note that $I = \Gamma_{\m_1\m_2\cdots \m_s}(H^c_P(M))$ is a $A_n(K)$-submodule of $H^c_P(M)$ and so $L \cong H^c_P(M)/I$ is a holonomic $A_n(K)$-module.

Let $f \in R\setminus P$. Recall $\injdim M = c$. We have an exact sequence
\[
H^c_{(P,f)}(M) \rt H^c_P(M) \rt H^c_P(M)_f \rt H^{c+1}_{(P,f)}(M) = 0.
\]
Thus the natural map $\eta \colon L \rt L_f$ is surjective. As $f \notin P$ and $\Ass L = \{ P \}$ we get that $\eta $ is also injective. Thus $L = L_f$ for every $f \in R\setminus P$. It follows that $L = L_P$. Also note that $L_P = H^c_P(M)_P$. By \cite[3.4(b)]{Lyu-1}, $L_P = E(R/P)^l$ for some finite $l > 0$. Thus we have that $E(R/P)$ is a holonomic $A_n(K)$-module. By \ref{lemma1} this is a contradiction.

\textit{Case 2:} $\height P \leq n-2$. \\ Suppose if possible $\mu_0(P, H^c_P(M)) \neq 0$.  Let $\Ass H^c_P(M) = \{ P, Q_1,\ldots Q_c \}$ where $Q_i \neq P$. As $H^c_P(M)$ is $P$-torsion we have that $Q_i \supsetneq P$ for all $i$. Let $f_i \in Q_i \setminus P$. Put $f = f_1\cdots f_c$. By Theorem \ref{locus}(4) there exists a maximal ideal $\m$ of $R$ such that $f \notin \m$ and $(R/P)_\m $ is regular local. Localize at $\m$. Notice $\Ass_{R_\m} H^c_P(M)_\m = \{ PR_\m \}$.
Let $g \in R_\m \setminus PR_\m$. Notice $\injdim_{R_\m} M_\m  \leq c$. So we have an exact sequence
\[
H^c_{(PR_\m, g)}(M_\m) \rt H^c_{PR_\m}(M_\m) \rt H^c_{PR_\m}(M_\m)_g \rt H^{c+1}_{(PR_\m, g)}(M_\m) = 0. 
\]
Thus the natural map $\eta \colon H^c_{PR_\m}(M_\m) \rt H^c_{PR_\m}(M_\m)_g$ is surjective. By Lemma \ref{local} it is also injective as $\Ass H^c_{PR_\m}(M_\m) = \{ PR_\m \}$. It follows that $H^c_{PR_\m}(M_\m) = H^c_{PR_\m}(M_\m)_g$.
So $H^c_{PR_\m}(M_\m)= H^c_{PR_\m}(M_\m)_{PR_\m}$. By \cite[1.4, 3.4(b)]{Lyu-1}, we get that $H^c_{PR_\m}(M_\m)_P \cong E(R_\m/PR_\m)^s$ for some finite $s > 0$. 
By \ref{flat-L} there exist a Lyubeznik functor $\FF^\prime$ on $Mod(R_\m)$ with 
$\FF^\prime(R_\m) = \FF(R)\otimes R_\m = M_\m$.   Observe that 
$\GG = H^c_{PR_\m}\circ \FF^\prime$ is a Lyubeznik functor on $R_\m$. We have $\GG(R_\m) = E(R_\m/PR_\m)^s$. This contradicts Lemma \ref{lemma2}.
\end{proof}

\section{Proof of Proposition \ref{main-2}}
In this section we prove Proposition \ref{main-2}. Throughout $K$ is an algebraically  closed field of characteristic zero. Let $R = K[X_1,\ldots, X_n]$ and let  $A_n(K)$ be the $n^{th}$-Weyl algebra over $K$. We use notions developed in \cite[Chapter 1]{B}, in particular we use the notion of Bernstein filtration of $A_n(K)$, good filtration, multiplicity and dimension of a finitely generated $A_n(K)$-module. 
We will use the fact that for any holonmic module $M$ we have $\ell(M) \leq e(M)$; here $\ell(M)$ denotes the length of $M$ as an $A_n(K)$-module and $e(M)$ denotes its multiplicity.

The following result is well-known. So we just sketch an argument.
\begin{proposition}\label{mult-inj}
Let $\m$ be a maximal ideal of $R$. Then $e(E(R/\m)) = 1$. In particular $E(R/\m)$ is a simple $A_n(K)$-module.
\end{proposition}
\begin{proof}(Sketch)
As $K$ is algebraically closed $\m = (X_1-a_1,\ldots,X_n-a_n)$ for some $a_1,\ldots,a_n \in K$. After a change of variables we may assume $a_1= \cdots = a_n = 0$.
Note $E(R/\m) = K[\partial_1,\ldots,\partial_n]$. The obvious filtration on $E(R/\m)$ is compatible with the Bernstein filtration and is good. So $e(E(R/\m)) = 1$.
\end{proof}

\s\label{bound} Let $M$ be a holonomic $A_n(K)$-module. Let $f \in R$ be a polynomial of degree $d$. Then by proof of Theorem 5.19 in Chapter 1 of \cite{B} we have 
$$e(M_f) \leq e(M)(1+ \deg f )^n.$$

We now give
\begin{proof}[Proof of Proposition \ref{main-2}]
Set $M = \FF(R)$.  Let $\m = (X_1-a_1,\ldots,X_n-a_n)$ be a maximal ideal of $R$. Fix $i$ with $0 \leq i \leq c$.  Notice $\mu_i(\m, M) = \mu_0(\m, H^i_\m(M))$, see \cite[1.4, 3.4(b)]{Lyu-1}.  If $H^i_\m(M) = E(R/\m)^{r_i}$ then $\mu_i(\m, M) = r_i =  \ell(H^i_\m(M))$.
To compute $H^i_\m(M)$ we use the \v{C}ech-complex:
\[
\C \colon 0 \rt M \rt \bigoplus_{j = 1}^{n}M_{(X_j-a_j)} \rt \cdots \rt M_{(X_1-a_1)\cdots(X_n-a_n)} \rt 0.
\]
In particular we have that $\ell(H^i_\m(M)) \leq \ell(\C^i)$. Notice $C^i$ has $\binom{n}{i}$ copies of modules of the form $M_f$ were $f$ is a product of $i$ distinct polynomials among $X_1-a_1,\cdots,X_n-a_n$. In particular $\deg f = i$. So by 
\ref{bound} we have $e(M_f) \leq e(M)(1+ i)^n$. Thus
$$r_i \leq e(\C^i) \leq \binom{n}{i}e(M)(1+ i)^n.$$
\end{proof}

\begin{remark}\label{char-p}
If $K_p$ is an algebraically closed field of characteristic $p$ and $S = K_p[X_1,\ldots,X_n]$ then Proposition \ref{main-2} holds for functors of the form 
$$\GG(-) = H^{i_1}_{I_1}(H^{i_2}_{I_2}(\cdots (H^{i_r}_{I_r}(-))\cdots ).$$
 The point is that $\GG(R)$ is holonomic $D$-module where $D$ is the ring of 
 $K_p$-linear differential operators over $S$. Here we use the notion of holonomicity by V. Bavula \cite{VB}. In this case
the bound $\ell(M_f) \leq n!\ell(M)(1+\deg f)^n$ holds, see \cite[Proof of 3.6]{Lyu-3}. The proof then follows by the same argument as before.

\end{remark}

\end{document}